\documentclass[12pt]{amsart}
\usepackage{geometry}
\usepackage{graphicx}
\usepackage{amssymb}
\usepackage{epstopdf}
\usepackage{amsmath,amscd}
\usepackage{amsthm}
\usepackage{url,verbatim}

\RequirePackage[colorlinks,citecolor=blue,urlcolor=blue]{hyperref}
\usepackage{breakurl}
\theoremstyle{plain}
\newtheorem{theorem}{Theorem}
\newtheorem{definition}[theorem]{Definition}
\newtheorem{lemma}[theorem]{Lemma}
\newtheorem{proposition}[theorem]{Proposition}

\newtheorem{example}[theorem]{Example}

\newtheorem{remark}[theorem]{Remark}

\newcommand\es{\varnothing}

\newcommand\ol{\overline}
\newcommand\Aut{\mathrm{Aut}}
\newcommand \sA{\mathcal{A}}

\newcommand\sP{{\mathcal P}}
\newcommand\sH{{\mathcal H}}

\newcommand\ZZ{{\mathbb Z}}
\newcommand\NN{{\mathbb N}}
\newcommand\QQ{{\mathbb Q}}
\renewcommand\a{\alpha}
\newcommand\om{\omega}
\newcommand\g{\gamma}

\newcommand\be{\beta}
\newcommand\si{\sigma}
\newcommand\eps{\epsilon}
\renewcommand\th{\theta}
\newcommand\De{\Delta}
\newcommand\qq{\qquad}
\newcommand\q{\quad}
\newcommand\resp{respectively}
\newcommand\spz{\mathrm{span}}

\newcommand\oo{\infty}
\newcommand\sG{{\mathcal G}}

\newcommand\Ga{\Gamma}
\newcommand\Si{\Sigma}
\newcommand\de{\delta}
\newcommand\id{{\bf 1}}

\newcommand\pc{p_{\text{\rm c}}}

\renewcommand\ell{l}

\newcommand\Tc{T_{\text{\rm c}}}
\newcommand\zc{z_{\text{\rm c}}}

\newcommand\olV{\ol V}
\newcommand\olE{\ol E}

\newcommand\Stab{\textrm{Stab}}

\newcommand\pd{\partial}

\newcommand\sAz{\sH}

\newcommand\ghf{graph height function}
\newcommand\ughf{unimodular graph height function}
\newcommand\GHF{group height function}
\newcommand\hdi{$\sH$-difference-invariant}
\newcommand\normal{\trianglelefteq}
\newcommand\what{\widehat}

\newcounter{mycount}

\newenvironment{numlist}{\begin{list}{\arabic{mycount}.}%
   {\usecounter{mycount}\labelwidth=1cm\itemsep 0pt}}{\end{list}}
\newenvironment{letlist}{\begin{list}{\rm(\alph{mycount})}%
   {\usecounter{mycount}\labelwidth=1cm\itemsep 0pt}}{\end{list}}

\numberwithin{equation}{section}
\numberwithin{theorem}{section}
\numberwithin{figure}{section}

\title[Locality of connective constants]{Locality of connective constants}
\author{Geoffrey R.\ Grimmett}
\address{Statistical Laboratory, Centre for
Mathematical Sciences, Cambridge University, Wilberforce Road,
Cambridge CB3 0WB, UK}
\email{g.r.grimmett@statslab.cam.ac.uk}
\urladdr{\url{http://www.statslab.cam.ac.uk/~grg/}}
\author{Zhongyang Li}
\address{Department of Mathematics,
University of Connecticut,
Storrs, Connecticut 06269-3009, USA} 
\email{zhongyang.li@uconn.edu}
\urladdr{\url{http://www.math.uconn.edu/~zhongyang/}}

\begin{document}

\begin{abstract}
The connective constant $\mu(G)$ of a quasi-transitive graph $G$
is the exponential growth rate of the number of self-avoiding walks 
from a given origin. We prove a locality theorem
for connective constants, namely, that the connective constants
of two graphs are close in value whenever the graphs agree
on a large ball around the origin (and a further condition is satisfied). 
The proof is based on a generalized bridge
decomposition of self-avoiding walks,
which is valid subject to the assumption that the
underlying graph is quasi-transitive and possesses a so-called 
\emph{\ughf}. 

\end{abstract}

\date{29 November 2014, revised 14 July 2018} 

\keywords{Self-avoiding walk, connective constant, 
vertex-transitive graph, quasi-transitive graph, bridge decomposition, Cayley graph, unimodularity}

\subjclass[2010]{05C30, 82B20}

\maketitle

\section{Introduction, and summary of results}

There is a rich theory of interacting systems on infinite graphs.
The probability measure governing a  process has, typically, a continuously varying
parameter, $z$ say, and there is a singularity at some 
\lq critical point' $\zc$. 
The numerical value of 
$\zc$ depends in general on the
choice of underlying graph $G$, and a significant
part of the associated literature is directed towards estimates
of $\zc$ for different graphs. In most cases of interest, 
the value of $\zc$ depends on more than the geometry of some bounded domain only.  
This observation provokes the question of \lq locality': to what degree
is the value of $\zc$ determined by the knowledge of a bounded
domain of $G$?

The purpose of the current paper is to present a locality theorem
(namely, Theorem \ref{thm2}) for the connective constant $\mu(G)$ of the graph $G$.
A self-avoiding walk (SAW) is a path that visits no 
vertex more than once. 
SAWs were introduced in the study of long-chain polymers 
in chemistry (see, for example, the 1953 volume
of Flory, \cite{f}), and their theory has been much developed since 
(see the book of Madras and Slade, \cite{ms}, 
and the recent review \cite{bdgs}). If the underlying 
graph $G$ has some periodicity,  the number of $n$-step SAWs 
from a given origin grows  exponentially
with some growth rate $\mu(G)$ called the 
\emph{connective constant} of the graph $G$. 

There are only few graphs $G$ for which the numerical value of
$\mu(G)$ is known exactly (detailed references for a number of such cases may be found in \cite{GrL1}), 
and a substantial part of
the literature on SAWs is devoted to inequalities
for  $\mu(G)$. The
current paper may be viewed in this light, as a continuation of the
series of papers on the broad topic of connective constants of transitive graphs by
the same authors, see  \cite{GrL2, GrL3, GrL1,GrLrev2}.

The main result (Theorem \ref{thm2}) of this paper is as follows. 
Let $G$, $G'$ be infinite, vertex-transitive graphs,
and write $S_K(v,G)$ for the $K$-ball around the vertex $v$ in $G$.
If $S_K(v,G)$ and $S_K(v',G')$
are isomorphic as rooted graphs, then 
\begin{equation}\label{eq:loc1}
|\mu(G)-\mu(G')| \le \eps_K(G),
\end{equation}
where $\eps_K(G)\to 0$ as $K \to\oo$.  
(A related result holds for quasi-transitive graphs.)
This is proved subject to certain conditions on the graphs $G$, $G'$,
of which the primary condition is that they support so-called \lq \ughf s'
(see Section \ref{sec:sec2} for the definition of a \ghf).
The existence of \ughf s permits the use of a \lq bridge
decomposition' of SAWs (in the style of the work of
Hammersley and Welsh \cite{HW62}), and this leads in turn
to computable sequences that converge to $\mu(G)$ from above
and below, \resp. The locality result of \eqref{eq:loc1}
may be viewed as a partial answer to a question of
Benjamini, \cite[Conj.\ 2.3]{Ben13}, made 
independently of the work reported here.

A class of vertex-transitive graphs of special interest is provided
by the Cayley graphs of finitely generated groups.  Cayley
graphs have algebraic as well as geometric structure, and this
allows a deeper investigation of locality and of graph height functions.
The corresponding investigation is reported in the companion paper \cite{GL-Cayley} where,
in particular, we present a method for the construction  of a \ghf\
via a suitable harmonic function on the graph.

The locality question for percolation
was approached by Benjamini, Nachmias, and Peres
\cite{bnp} for tree-like graphs. Let $G$ be vertex-transitive with degree $d+1$. It is elementary that the percolation
critical point satisfies $\pc\ge 1/d$ (see
\cite[Thm 7]{BH57}), and an
asymptotically equivalent  upper bound for $\pc$
was developed in \cite{bnp} for a certain family of graphs which are (in a certain
sense) locally tree-like.
In recent work of Martineau and Tassion \cite{MT}, a
locality result has been proved for percolation on abelian graphs.
The proof extends the methods and conclusions of \cite{gm}, 
where it is proved that the slab critical points converge to 
$\pc(\ZZ^d)$, in the limit as the slabs become infinitely `fat'.
(A related result for connective constants is included here at Example \ref{ex:torus}.)

We are unaware of a locality theorem for the critical temperature $\Tc$
of the Ising model. Of the potentially relevant work on the Ising model  to
date, we mention \cite{Bod,ch, Dob70, DobS, zl, Wei}.

This paper is organized as follows. Relevant background and notation is
described in Section \ref{sec:not}.
The concept of a \ghf\  is presented in 
Section \ref{sec:sec2},
where examples are included of infinite graphs with \ghf s.
Bridges and the bridge constant are defined in Section \ref{sec:saws}, 
and it is proved in Theorem \ref{thm1} that
the bridge constant equals the connective constant whenever there
exists a \ughf. The main \lq locality theorem' is given at 
Theorem \ref{thm2}.  Theorem \ref{thm4} is an application of
the locality theorem in the context of a sequence of quotient graphs;
this parallels the Grimmett--Marstrand theorem
\cite{gm} for percolation on (periodic) slabs, but with the underlying lattice replaced by a  
transitive graph with a \ughf.
Sections \ref{sec:pf1} and \ref{sec:pf1b} contain the proofs of Theorem \ref{thm1}.

\section{Notation and background}\label{sec:not}

The graphs $G=(V,E)$ considered here are generally assumed
to be infinite, connected, locally finite, undirected, and also
simple, in that they have neither loops nor multiple edges.
An  edge $e$ with endpoints $u$, $v$ is 
written as $e=\langle u,v \rangle$. 
If $\langle  u,v \rangle \in E$, we call $u$ and $v$ \emph{adjacent},
and we write $u \sim v$. The set of neighbours of $v$ is written as 
$\pd v=\{u\in V: \langle u,v\rangle \in E\}$. 

Loops and multiple edges have been excluded for 
cosmetic reasons only.
A SAW can traverse no loop, and thus loops may 
be removed without changing the connective constant.
The same proofs are valid in the presence of multiple edges.
When there are multiple edges, we are effectively
considering SAWs
on a weighted simple graph, and indeed our results are valid
for edge-weighted graphs with strictly positive weights, and for
counts of SAWs in which the contribution of a given SAW is
the product of the weights of the edges therein. 

The \emph{degree} of vertex $v$ is the number of edges
incident to $v$, denoted $\deg_G(v)$ or $\deg (v)$,
and $G=(V,E)$ is called \emph{locally finite} if
every vertex-degree is finite.  
The maximum vertex-degree is denoted 
$\de_G=\sup\{\deg_G(v): v \in V\}$.
The \emph{graph-distance} between two vertices $u$, $v$ is the number of edges
in the shortest path from $u$ to $v$, denoted $d_G(u,v)$.
We denote by $S_k(v)=S_k(v,G)$ the ball of $G$ with 
centre $v$ and radius $k$.

The automorphism group of the graph $G=(V,E)$ is
denoted $\Aut(G)$. A subgroup $\Ga \le \Aut(G)$ is said to \emph{act
transitively} on $G$  (or on its vertex-set $V$)
if, for $v,w\in V$, there exists $\g \in \Ga$ with $\g v=w$.
It is said to  \emph{act quasi-transitively} if there is a finite
set $W$ of vertices 
such that, for $v \in V$, there exist
$w \in W$ and $\g \in \Ga$ with $\g v =w$.
The graph is called \emph{(vertex-)transitive} 
(\resp, \emph{quasi-transitive}) if $\Aut(G)$ acts transitively
(\resp, quasi-transitively). 
For a subgroup $\sH \le \Aut(G)$ and a vertex $v \in V$, the orbit of $v$ under $\sH$ is written 
$\sH v$. The number of such orbits is written
as $M(\sH) =|G/\sH|$.

A \emph{walk} $w$ on $G$ is
an (ordered) alternating sequence $(w_0,e_0,w_1,e_1,\dots, e_{n-1}, w_n)$ of vertices $w_i$
and edges $e_i=\langle w_i, w_{i+1}\rangle$, with $n \ge 0$.
We write $|w|=n$ for the \emph{length} of $w$, that is, the number of edges in $w$.
The walk $w$ is called \emph{closed} if $w_0=w_n$. We note that $w$ is
directed from $w_0$ to $w_n$.
When, as generally assumed, $G$ is simple, we may abbreviate $w$ to the sequence
$(w_0,w_1,\dots, w_n)$ of vertices visited. 

A \emph{cycle} is a closed walk $w$ traversing three or more distinct edges, and 
satisfying $w_i\ne w_j$ for $1 \le i < j \le n$. Strictly speaking,
cycles (thus defined) have orientations derived from the underlying walk, and for
this reason we may refer to them sometimes as \emph{directed}  cycles of $G$.

An \emph{$n$-step self-avoiding walk} (SAW) 
on $G$ is  a walk containing $n$ edges
no vertex of which appears more than once.
Let $\Si_n(v)$ be the set of $n$-step SAWs starting at $v$, with
cardinality $\si_n(v):=|\Si_n(v)|$, and let
\begin{equation}\label{eq:defssup}
\si_n=\si_n(G) := \sup\{\si_n(v): v \in V\}.
\end{equation}
We have in the usual way (see \cite{hm,ms}) that
\begin{equation}\label{eq:sisub}
\si_{m+n} \le \si_m\si_n,
\end{equation}
whence the \emph{connective constant}
$$
\mu=\mu(G) :=\lim_{n\to\oo} \si_n^{1/n}
$$
exists, and furthermore
\begin{equation}\label{eq:silb}
\si_n \ge \mu^n, \qq n\ge 0.
\end{equation}
Hammersley \cite{jmhII} proved that, if $G$ is quasi-transitive, 
\begin{equation}\label{connconst}
\lim_{n \to \oo} \si_n(v)^{1/n} = \mu, \qq v \in V.
\end{equation}

We select a vertex of $G$ and call it the \emph{identity} or \emph{origin}, denoted $\id$.
Further notation concerning SAWs will be introduced when needed. The concept of a
\lq \ghf'  is explained in the next section, and that
leads  in Section \ref{sec:saws} to the definition of a \lq bridge'.

Let $\Ga\le \Aut(G)$ act transitively, and let $\sH\le\sA$.
We define the (simple) \emph{quotient graph} 
$G/\sH=(\olV,\olE)$ as follows.   The vertex-set $\olV$  comprises the orbits $\ol v := \sAz v$ 
as $v$ ranges over $V$. For $v,w \in V$, 
we place an edge between $\ol v$ and $\ol w$ if and only if $\pd v\cap \ol w \ne \es$
(if $\ol v = \ol w$, such an edge is a loop).  
Further detaills of quotient graphs may be found in, for example, \cite[Sect.\ 3.4]{GrL3}.

The set of integers is written as $\ZZ=\{\dots,-1,0,-1,\dots\}$, 
the natural numbers as
$\NN=\{1,2,3\dots\}$, and the rationals as $\QQ$.

\section{Quasi-transitive graphs and \ghf s}\label{sec:sec2}

Let $G=(V,E)$ be
an infinite, connected, quasi-transitive, locally finite, simple graph. 

\begin{definition} \label{def:height}
A \emph{\ghf} on $G$ is a pair $(h,\sH)$ such that:
\begin{letlist}
\item $h:V \to\ZZ$, and $h(\id)=0$, 
\item $\sH$ is a subgroup of $\Aut(G)$ acting quasi-transitively on $G$ 
such that $h$ is \emph{\hdi} in the sense that
$$
h(\a v) - h(\a u) = h(v) - h(u), \qq \a \in \sH,\ u,v \in V,
$$
\item for  $v\in V$,
there exist $u,w \in \pd v$ such that
$h(u) < h(v) < h(w)$.
\end{letlist}
A \emph{\ughf} is a \ghf\ $(h,\sH)$ with the action of
$\sH$ being unimodular.
\end{definition}

The expression `\ghf' is in contrast to the `\GHF' of \cite{GL-Cayley}. It is
explained in \cite{GL-Cayley} that a \GHF\ of a finitely 
generated group is a (unimodular) \ghf\ 
on its Cayley graph, but not necessarily \emph{vice versa}.

We remind the reader of the definition of unimodularity.
Let $G=(V,E)$ be an infinite graph and $\sH \le \Aut(G)$.
The ($\sH$-)\emph{stabilizer} 
 $\Stab_v$ ($=\Stab_v^\sH$) of $v\in  V$ is the set of $\g\in \sH$ for which $\g(v) = v$.
The group $\sH$ is said to \emph{act freely} if $\Stab_v$ contains only
the identity map.
The action of $\sH$ (or the group $\sH$ itself)
is called \emph{unimodular} if and only if 
\begin{equation}\label{g804}
|\Stab_u v| = |\Stab_v u|, \qquad v\in V,\ u \in \sH v.
\end{equation} 
Further details of unimodularity may be found in \cite[Chap.\ 8]{LyP}. 
The assumption of unimodularity is necessary
in the Bridge Theorem \ref{thm1} (see Remark \ref{ex:nec-tree}). 

Associated with a \ghf\ $(h,\sH)$
are two integers $d$, $r$ which 
will play roles in the following sections and which we define next.
Let
\begin{equation}\label{eq:defd}
d=d(h)=\max\bigl\{|h(u)-h(v)|: u,v\in V,\ u \sim v\bigr\}.
\end{equation}

If $\sH$ acts transitively, we set
$r=0$. Assume $\sH$ does not act transitively, and 
let $r=r(h,\sH)$ be the infimum of all $r$ such that the following holds.
Let $o_1,o_2,\dots,o_M$ be representatives of the orbits of $\sH$.
For $i\ne j$, there
exists $v_j \in \sH o_j$ 
such that $h(o_i)<h(v_j)$, and a SAW $\nu(o_i,v_j)$ from $o_i$
to $v_j$, with length $r$ or less, all of whose vertices $x$, other than
its endvertices,  satisfy $h(o_i)<h(x)< h(v_j)$.
We fix such a SAW, and denote it as above. We set $\nu(o_i,o_i)=\{o_i\}$.
Such SAWs will be used in Section \ref{sec:pf1b}. 
The following proposition is proved at the end of this section.

\begin{proposition}\label{lem:rfin}
Let $(h,\sH)$ be a \ghf\  on the graph $G$. Then
$r=r(h,\sH)$ satisfies
\begin{gather}
0\le r \le (M-1)(2d+1)+2,\label{new41}\\
M \le |S_r(\id,G)| \le \frac{\de_G^{r+1}-1}{\de_G-1},
\label{new42}
\end{gather}
where $M=|G/\sH|$ and $d$ is given by \eqref{eq:defd}.
\end{proposition}

Not every quasi-transitive graph has a \ghf. For example, condition (c), above, fails
if $G$ has a cut-vertex whose removal breaks $G$ into an infinite and a finite part.
We do not have a useful necessary and sufficient 
condition for the existence of a \ghf. 

\begin{remark}\label{rem1}
There exist infinite Cayley graphs that support no \ghf.
Examples are provided in the paper \cite{GL-amen}, which postdates
the current work.
\end{remark}

Here are several examples of transitive graphs with \ghf s.
\begin{letlist}
\item  The \emph{hypercubic lattice} $\ZZ^n$ 
with, say,  $h(x_1,x_2,\dots,x_n) = x_1$. With $\sH$ the set 
of translations of $\ZZ^n$, we have $d(h)=1$ and $r(h)=0$.

\item
The \emph{$2n$-regular tree} $T$ with $n \ge 1$ is the Cayley graph
of the free group $\Ga$ with $n$ generators $a_1, a_2, \dots, a_n$.
Each vertex $v$ is labelled as a product of the form 
$a_{i_1}^{j_1}a_{i_2}^{j_2}\cdots a_{i_m}^{j_m}$, where 
$j_r \in \{-1,1\}$.  We set $h(v)$ equal to the sum of the $j_r$ such that
$i_r=1$, and let $\sH$ be the action of $\Ga$ by left-multiplication.
Since $\Ga$ acts freely, it is unimodular. We have $d=1$ and $r=0$.

A similar construction applies to the $n$-regular tree with $n\ge 3$ odd, by
choosing $a_1$ to be a generator with infinite order.  

\begin{figure}[h]
\centering
\includegraphics[width=0.5\textwidth]{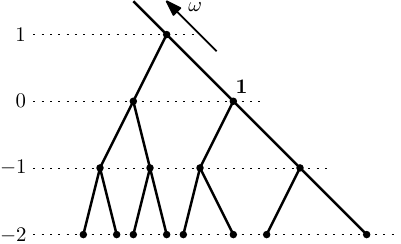}
\caption{The $3$-regular tree with the \lq horocyclic' height
function.}\label{fig:tree}
\end{figure}

Let $m \ge 3$.
The $m$-regular tree $T$ possesses also a \emph{non-unimodular} \ghf,
as follows. Let $\om$ be a ray of $T$, 
and `suspend' $T$ from $\om$ as in Figure \ref{fig:tree}. A given vertex
on $\omega$  is designated as identity $\id$ and is
given height $0$, and other vertices have `horocyclic' heights as indicated
in the figure.
The set $\sH$ is the subgroup of automorphisms that fix 
the end of $T$ determined by $\om$,
and $\sH$ acts transitively but is not unimodular
(see \cite{LyP,Trof}). 
We have $d=1$ and $r=0$.

Using the last pair $T$, $\sH$, we may construct a graph possessing  a \ghf\ but no \ughf.
Consider the `grandparent' graph $G$ derived from $T$ by adding an edge 
between each vertex and its grandparent in the direction of $\om$ (see \cite{Trof}).
Its automorphism group may be taken as $\sH$. Suppose $G$ has a \ughf\ $(h,\sA)$.
By \cite[Cor.\ 8.11, Prop.\ 8.12]{LyP} and the fact that $\sA \le \sH$, we have that
$\sH$ is unimodular, which is not true. 

\begin{remark}\label{rem:nonunim}
More generally, if $\sA$ is a quasi-transitive, unimodular group of
automorphisms on a graph $G$, and $\sA\le\sH\le\Aut(G)$, then
$\sH$ is unimodular. In particular, if $\Aut(G)$ is non-unimodular,
then $G$ has no \ughf.
\end{remark} 

\item There follow three examples
of Cayley graphs of finitely presented groups
(readers are referred to \cite{GL-Cayley} for further information
on Cayley graphs). 
The \emph{discrete Heisenberg group}
\begin{equation*}
\Ga=\left\{\left(\begin{array}{ccc}1&x&z\\0&1&y\\0&0&1\end{array}\right):x,y,z\in\ZZ\right\},
\end{equation*}
has 
generator set $S=\{s_1,s_2,s_3,s_1',s_2',s_3'\}$ where
\begin{equation*}
s_1=\left(\begin{array}{ccc}1&1&0\\0&1&0\\0&0&1\end{array}\right),\q s_2=\left(\begin{array}{ccc}1&0&0\\0&1&1\\0&0&1\end{array}\right),\q s_3=\left(\begin{array}{ccc}1&0&1\\0&1&0\\0&0&1\end{array}\right),
\end{equation*}
and relator set  
$$
R=\{s_1s_1', s_2s_2', s_3s_3'\}\cup
\{s_1s_2s_1's_2's_3',s_1s_3s_1's_3',
s_2s_3s_2's_3'\}.
$$ 
Consider its Cayley graph.
To a directed edge of the form $[ v,vs_1\rangle$
(\resp, $[v, vs_1'\rangle$)
we associate the height difference $1$ (\resp, $-1$),
and to all other edges height difference $0$. 
 The height  $h(v)$ of vertex $v$
is given by adding the height differences along any
directed path
from the identity $\id$ to $v$, which is to say that
$$
h\left[\left(\begin{array}{ccc}1&x&z\\0&1&y\\0&0&1\end{array}\right)\right]
=x.
$$
The function $h$ is well defined because the sum of the height differences around any cycle
arising from a relator is $0$.
We take $\sH$ to be the Heisenberg group, acting by left-multiplication, 
and we have $d=1$ and $r=0$.

\item The \emph{square/octagon lattice} of Figure \ref{fig:so}  
is the Cayley graph of the group with generators $s_1$, $s_2$, $s_3$
and relators $\{s_1^2,s_2^2,s_3^2,s_1s_2s_1s_2, s_1s_3s_2s_3s_1s_3s_2s_3\}$.
It has no \ghf\  with $\sH$ acting \emph{transitively}. 
There are numerous ways to define a height
function with quasi-transitive $\sH$, of which we mention one.
Let $\sH$ be the automorphism subgroup generated by the
shifts that map $\id$ to $\id'$ and $\id''$, \resp, and let the \ghf\    
be as in the figure.
Since $\sH$ acts freely, it is unimodular. We have $d=1$ and $r=6$.

\begin{figure}[h]
\centerline{\includegraphics[width=0.5\textwidth]{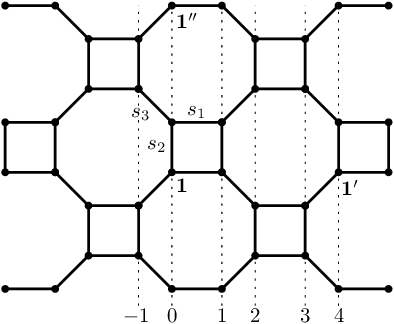}}
\caption{The square/octagon lattice. The subgroup $\sH$ is 
generated by the shifts $\tau'$, $\tau''$ that map $\id$ to $\id'$ and
$\id''$, \resp, and the heights of vertices are as marked.}
\label{fig:so}
\end{figure}

\item
The \emph{hexagonal lattice} of Figure \ref{fig:hcl}
is the Cayley graph of a finitely presented group.
It possesses a \ughf\  $h$ with $\sH$ acting quasi-transitively, 
as follows. Let $\sH$ be the set of automorphisms of
the lattice that act by translation of the figure, and let
the heights be as 
given in the figure. 
We have $d=1$ and $r=1$.

\begin{figure}[htbp]
\centering
\includegraphics[width=0.55\textwidth]{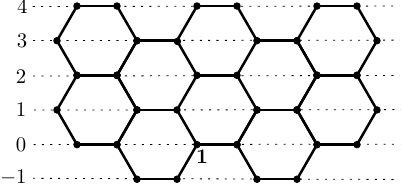}
\caption{The hexagonal lattice. The heights of vertices are 
as marked.}\label{fig:hcl}
\end{figure}

\item The \emph{Diestel--Leader graphs} DL$(m,n)$ with $m,n\ge 2$ and 
$m \ne n$ were proposed in \cite{DL}
as candidates for transitive
graphs that are quasi-isometric to no Cayley graph,
and this conjecture was proved in \cite{AFW}
(see \cite{Dun11} for a   further example).
They arise through a certain combination (details of which are omitted here)
of an $(m+1)$-regular tree and an $(n+1)$-regular tree.
The horocyclic \ghf\ of either tree provides a 
\ghf\ for the combination, which
is unimodular if and only if $m=n$.  
When  $m\ne n$, by Remark \ref{rem:nonunim},
there exists no unimodular \ghf. 
\end{letlist}

\begin{proof}[Proof of Proposition \ref{lem:rfin}]
Let $(h,\sH)$ be as in Definition \ref{def:height}, and assume that $\sH$ acts quasi-transitively but 
not transitively. For $v,w\in V$, we write $v \to w$ if there exist $v'\in \sH v$,
$w'\in \sH w$ such that (i) $h(v')<h(w')$,
and (ii) there is a SAW $\nu=(\nu_0,\nu_1,\dots, \nu_m)$ with $\nu_0= v'$, $\nu_m=w'$, and
$h(v')<h(\nu_j) < h(w')$ for $1\le j < m$.
We prove next that $v \to w$ for all pairs $v$, $w$ lying in distinct orbits of $V$.
Since $\sH$ has only finitely many orbits, this will
imply that $r(h,\sH)<\oo$.

Let $u\in V$,  and let $T_u$ be a sub-tree of $G$ containing $u$ and
exactly one representative of each orbit of $\sH$. (The tree $T_u$
may be obtained as a lift of a spanning tree of the quotient graph 
$G/\sH$.) With
$M=|G/\sH|$, the tree $T_u$ has $M-1$ edges.
Let
$$
\De_u=\max\bigl\{|h(a)-h(b)|: a,b \in V(T_u)\bigr\},
$$
where $V(T_u)$ is the vertex-set of $T_u$. By \eqref{eq:defd}, 
\begin{equation}\label{eq:Dupper}
|\De_u| \le (M-1)d, \qq u \in V.
\end{equation}
 
By Definition \ref{def:height}(c), for $v \in V$, we may pick a doubly 
infinite SAW $\pi(v)
=(\pi_j(v): j \in \ZZ)$
with $\pi_0(v)=v$, such that $h(\pi_j(v))$ is strictly 
increasing in $j$. Since $h$ takes integer values,
\begin{equation}\label{eq:eps}
h(\pi_{j+1}(v))-h(\pi_j(v)) \ge 1, \qq j\in \ZZ,\ v \in V.
\end{equation}

Let $v,w \in V$ be in distinct orbits of $\sH$. Let $v'=\pi_R(v)$
and $w'=\pi_{-R}(w)$ where $R\ge 1$ will be chosen soon. 
Find $\a\in\sH$ such that
$\a v' \in V(T_{w'})$. Let $\nu$ be the walk obtained
by following the sub-SAW of
$\a\pi(v)$ from $\a v$ to $\a v'$, followed
by the sub-path of $T_{w'}$ from $\a v'$ to $w'$,
followed by the sub-SAW of $\pi(w)$ from $w'$ to
$w$. The length of $\nu$ is at most  $2R+M-1$.

By \eqref{eq:eps}, we can pick $R$ sufficiently large that
\begin{align*}
h(\a v) &< \min\{h(a): a \in V(T_{w'})\}\\ 
&\le h(\a v')\q \text{(\resp, $h(w')$)}\\
&\le \max\{h(a): a \in V(T_{w'})\} < h(w),
\end{align*}
and indeed, by \eqref{eq:Dupper}, it suffices that $R = (M-1)d+1$.
By loop-erasure of $\nu$, we obtain a SAW 
$\nu'=(\nu'_0,\nu_1',\dots,\nu'_m)$ with $\nu'_0=\a v$, $\nu'_m=w$, 
\begin{equation}\label{eq:bound}
m \le 2R+M-1 \le 2(M-1)d+2 +(M-1),
\end{equation}
and
$h(\nu_0') < h(\nu'_j) <h(\nu'_m)$
for $1\le j < m$. Therefore, $v \to w$ as required.
The bound \eqref{new41} follows from \eqref{eq:bound}.

Inequality \eqref{new42} is a consequence of the definition of $r(h,\sH)$.
\end{proof}

\section{Bridges and the bridge constant}\label{sec:saws}

Assume that $G$ is quasi-transitive with \ghf\  $(h,\sH)$. 
The forthcoming definitions depend on the choice of 
pair  $(h,\sH)$.

Let $v \in V$ and $\pi=(\pi_0,\pi_1,\dots,\pi_n)\in\Si_n(v)$.
We call $\pi$ a \emph{half-space SAW} if 
$$
h(\pi_0)<h(\pi_i), \qq 1\leq i\leq n,
$$
and we write $c_n(v)$ for the number of half-space walks
with initial vertex $v$.
We call $\pi$ a \emph{bridge} if 
\begin{equation}\label{eq:bridge}
h(\pi_0)<h(\pi_i) \leq h(\pi_n), \qq 1\leq i\leq n,
\end{equation}
and a \emph{reversed bridge} if \eqref{eq:bridge} is replaced by
$$
h(\pi_n)\le h(\pi_i) < h(\pi_0), \qq 1\leq i\leq n.
$$
 
The \emph{span} of a SAW $\pi$ is defined as 
$$
\spz(\pi) = \max_{0\leq i\leq n}h(\pi_i)-\min_{0\leq i\leq n}h(\pi_i).
$$
The number of $n$-step bridges from $v$ with span $s$ is denoted 
$b_{n,s}(v)$, and in addition
$$
b_n(v) = \sum_{s=0}^\oo b_{n,s}(v)
$$
is the total number of $n$-step bridges from $v$. Let
\begin{equation}\label{eq:defbinf}
b_n =b_n(G):= \min\{b_n(v): v \in V\}.
\end{equation}
It is easily seen (as in \cite{HW62}) that
\begin{equation}\label{eq:b-subadditive}
b_{m+n} \ge b_m b_n,
\end{equation}
from which we deduce the existence of the \emph{bridge constant}
\begin{equation}\label{eq:bexists}
\be  = \be(G) = \lim_{n\to\oo} b_n^{1/n}
\end{equation}
satisfying
\begin{equation}\label{eq:bless}
b_n \le \be^n, \qq n \ge 0.
\end{equation}

\begin{remark}\label{rem:hdep}
The bridge constant $\be$ depends on the choice
of \ghf. We shall see in Theorem \ref{thm1}
that its value is constant across the set of \emph{unimodular} \ghf s. 
\end{remark}

\begin{proposition}\label{prop:mequalsb}
Let  $G=(V,E)$ be an infinite, connected, quasi-transitive,
locally finite, simple graph possessing a \ghf\ $(h,\sH)$. Then
\begin{equation}\label{eq:bvconv}
b_n(v)^{1/n} \to \be, \qq v \in V,
\end{equation}
and furthermore 
\begin{equation}\label{eq:bnabove}
b_n(v) \le \be^{n+r}, \qq n\ge 1,\ v \in V,
\end{equation}
where $r=r(h,\sH)$ is given after \eqref{eq:defd}.
\end{proposition}

\begin{theorem}[Bridge theorem]\label{thm1}
Let $G=(V,E)$ be an infinite, connected, quasi-transitive,
locally finite, simple graph possessing a \ughf\ $(h,\sH)$. Then $\be = \mu$.
\end{theorem}

This theorem extends that of Hammersley and Welsh \cite{HW62}
for $\ZZ^d$, and has as corollary that the value of the bridge constant
is independent of the choice of pair $(h,\sH)$.
The proof of the theorem is deferred to Sections \ref{sec:pf1} and \ref{sec:pf1b}.

\begin{remark}\label{ex:nec-tree}
Here is an example of the necessity of unimodularity in 
Theorem \ref{thm1}. The $m$-regular tree $T$ of
Figure \ref{fig:tree} possesses a \ughf\ $h$, and also a non-unimodular, horocyclic
\ghf\ $h'$. The number of $n$-step bridges relative to $h'$ is exactly $1$,
so that $\be=1$ using $h'$. On the other hand, $\mu=2$.
\end{remark}

\begin{remark}\label{rem:an}
It is proved in \cite{GrL3} that, in certain situations,  the quotienting of a graph $G$
by a non-trivial subgroup of its automorphism group
leads to strict reduction in the value of its connective constant, and the
question is posed there of whether one can establish a 
concrete lower bound on the magnitude of the change
in value. It is proved
in \cite[Thm 3.11]{GrL3} that this can be done whenever
there exists a real sequence $(a_n)$ satisfying $a_n \uparrow \mu(G)$,
each element of which can be calculated in finite time. For
any transitive graph $G$ satisfying the hypothesis of Theorem
\ref{thm1}, we may take $a_n=b_n^{1/n}$.
\end{remark}

\begin{proof}[Proof of Proposition \ref{prop:mequalsb}]
Assume $G$ has \ghf\  $(h, \sH)$. If $G$ is transitive, the claim is trivial, 
so we assume $G$ is quasi-transitive but not transitive.
For $v,w \in V$ with $w \notin \sH v$,
let $\nu(v,w)$ be a SAW from $v$ to some $w'\in \sH w$ with $h(v)<h(w')$, every vertex $x$ of which,
other than its endvertices, satisfies $h(v) < h(x) < h(w')$. 
We may assume that the length $l(v,w)$ of $\nu(v,w)$
satisfies $l(v,w)\le r$ for all such pairs $v$, $w$.

Choose $x\in V$ such that $b_{n+r}(x)=b_{n+r}$. Let $l=l(x,v)$ if $x \notin \sH v$,
and $l=0$ otherwise.
Since $b_m(x)$ is non-decreasing in $m$, and $l \le r$, 
$$
b_n(v) \le b_{n+l}(x) \le b_{n+r}(x) =b_{n+r},
$$
and \eqref{eq:bnabove} follows by \eqref{eq:bless}.
The limit \eqref{eq:bvconv} follows by 
\eqref{eq:defbinf} and \eqref{eq:bexists}.
\end{proof}

\section{Locality of connective constants}\label{sec:loc}

Let $\sG$ be the class of infinite, connected, quasi-transitive, locally finite, simple, rooted graphs.
For $G \in \sG$, we label the root as $\id=\id_G$ and call it the 
\emph{identity} or \emph{origin} of $G$. 
The \emph{ball} $S_k(v)=S_k(v,G)$, with centre $v$ and radius $k$, is
the subgraph of $G$ induced by the set of its vertices within
graph-distance $k$ of $v$. For $G,G'\in \sG$, we write
$S_k(v,G) \simeq S_k(v',G')$ if there exists a  graph-isomorphism from $S_k(v,G)$ to
$S_k(v',G')$ that maps $v$ to $v'$.  We define the \emph{similarity}
of $G,G' \in \sG$ by
$$
K(G,G') = \max\bigl\{k: S_k(\id_G,G) \simeq S_k(\id_{G'},G')\bigr\}, \qq G,G' \in \sG,
$$
and the  distance-function $d(G,G') = 2^{-K(G,G')}$.
Thus $d$ defines a metric on $\sG$ quotiented by graph-isomorphism,
and this metric space was introduced by Babai \cite{Bab};
see also \cite{BenS01,DL}.

For integers $D\ge 1$ and $R\ge 0$, let $\sG_{D,R}$ be the set of
all $G\in\sG$ which possess a \ughf\  
$h$ satisfying $d(h)\le D$
and $r(h,\sH) \le R$. For a quasi-transitive graph $G$, we write 
$M(G)=|G/\Aut(G)|$ for the number of orbits under its automorphism group.
The locality theorem for quasi-transitive graphs follows,
with proof at the end of the section. 
The theorem may be regarded as a partial resolution of a question
of Benjamini, \cite[Conj.\ 2.3]{Ben13}, which was posed independently of the work reported here.

\begin{theorem}[Locality theorem for connective constants]
\label{thm2}
Let  $G \in \sG$.
Let $D\ge 1$ and $R \ge 0$, and let 
$G_n \in \sG_{D,R}$ for $n \ge 1$. If
$K(G,G_n)  \to \oo$ as $n \to\oo$, then
$\mu(G_n) \to \mu(G)$.
\end{theorem}

The following application of Theorem \ref{thm2} is prompted in part by a result in percolation
theory. Let $\pc(G)$ be the critical probability of either bond or 
site percolation on an infinite graph $G$, and let
$\ZZ^d$ be the $d$-dimensional hypercubic lattice
with $d \ge 3$, and 
$S_k=\ZZ^2\times \{0,1,\dots,k\}^{d-2}$.
It was proved by Grimmett and Marstrand \cite{gm} that
\begin{equation}\label{gm1}
\pc(S_k) \to \pc(\ZZ^d) \qq\text{as } k \to\oo.
\end{equation}
By Theorem \ref{thm2}(c) and the bridge construction of  Hammersley
and Welsh \cite{HW62}, the connective constants satisfy
\begin{equation}\label{gm2}
\mu(\what S_k) \to \mu(\ZZ^d) \qq\text{as } k \to \oo,
\end{equation}
where $\what S_k$ is obtained from $S_k$ by imposing  periodic
boundary conditions in its $d-2$ bounded dimensions. 
Such a limit may be extended as follows to more general situations.
For simplicity, we consider the case of \emph{transitive} graphs only.

Let $G\in\sG$ and let $\Ga$ be a subgroup of $\Aut(G)$ that acts transitively.
Let  $\sA$ be a normal subgroup of
$\Ga$, and assume that $\De(\sA):=d_G\bigl(\id,\sA\id\setminus\{\id\}\bigr)$ 
satisfies $\De \ge 3$. The group $\sA$ gives rise to a (simple) quotient
graph $G/\sA$ (see Section \ref{sec:not}). 
Since $\sA$ is a normal subgroup of $\Ga$,
$\Ga$ acts on $G/\sA$ (see \cite[Remark 3.5]{GrL3}), 
whence $G/\sA$ is transitive.

\begin{theorem}\label{thm4}
Let $G\in\sG$ and $D \ge 1$. Let $\Ga$ act transitively on $G$, and let $\sA_n\normal \Ga$ 
satisfy $\De(\sA_n) \to\oo$ as $n\to\oo$.
Assume  that $G_n:= G/\sA_n \in \sG_{D,0}$ for $n \ge 1$. 
Then $\mu(G_n) \to \mu(G)$ as $n\to\oo$.
\end{theorem}

\begin{proof}
The quotient graph $G_n$ is obtained from $G$ by identifying
any two vertices $v \ne w$ with $w =\a v$ and $\a\in\sA_n$.  
For such $v$, $w$, we have 
$d_G(v,w) \ge \De(\sA_n)$. 
Therefore,
$K(G,G_n) \ge \frac12 \De(\sA_n)-1$,
and the result follows by Theorem \ref{thm2}.
\end{proof}

\begin{example}\label{ex:torus}
Let $G$ be the hypercubic lattice $\ZZ^n$ with $n \ge 2$,
and let $\Ga$ be the group of its translations.
Choose $v=(v_1,v_2,\dots,v_n)\in \ZZ^n$ with $\|v\|:=\max_i |v_i|$
satisfying $\|v\|\ge 3$, and let
$\a_v\in \Ga$ be the translation $w \mapsto w+v$.  Let
$\sP_v$ be the set of non-zero integer vectors perpendicular to $v$,
and, for convenience,  choose $p=(p_1,p_2,\dots,p_n)\in\sP_v$ in such a way 
that $\|p\|$ is a minimum.
For $z\in \ZZ^n$, 
let 
$h_v(z)= z\cdot p$, so that $(h_v,\Ga)$ is a \ghf\  with
$d(h_v) = \|p\|$ and $r(h_v,\Ga)=0$. Since $\Ga$ acts freely, it
is unimodular.

Let $\sA_v$ be the subgroup of $\Ga$ generated 
by $\a_v$ (which is invariably normal). It may be seen that
$(h_v,\Ga/\sA_v)$ is a height function
for $G/\sA_v$ with $d$ and $r$ as above. 
In the notation of Theorem \ref{thm4}, we have that
$\mu(G/\sA_{v_m}) \to \mu(G)$ as $m \to\oo$, 
so long as the sequence $(v_m)$ satisfies
$|v_m|\to\oo$ and
$\limsup_{m\to\oo} d(h_{v_m}) < \oo$.
This may be regarded as a general version of the limit \eqref{gm2}.
\end{example}

We turn to the proof of Theorem \ref{thm2}, and present first a more detailed proposition.

\begin{proposition}
\label{prop2}
Let  $G \in \sG$.
\begin{letlist}
\item  Let $m \ge 1$.
There exists a non-increasing real sequence 
$(\eps_k:k\ge 1)$, depending on $G$ and $m$ only, and satisfying
$0<\eps_k \downarrow 0$ as $k \to \oo$,  such that,
for $G'\in \sG$ with $M(G')\le m$,
\begin{equation}
 \mu(G') \le \mu(G) + \eps_K,\label{ineq1}
 \end{equation}
whenever $K=K(G,G')$ satisfies $K \ge \max\{M(G),m\}$.
 
 \item Let $D,\De\ge 1$ and $R \ge 0$. There exists 
 $B=B(D,R,\De)\in(0,\oo)$ such that, for 
$G' \in \sG_{D,R}$ satisfying $\de_{G'} \le \De$,
 \begin{equation}\label{ineq2}
 \frac{\mu(G)}{f(K-L)}\le \be(G')
 = \mu(G'), \qq \text{if } K=K(G,G') > L,
 \end{equation}
where $f(x)=e^{B/\sqrt x}$ and
$$
L=\max\left\{M(G),\frac{\De^{R+1}-1}{\De-1}\right\}-1.
$$
\end{letlist} 
\end{proposition}

\begin{proof}
Let $G\in\sG$. 
Since the quotient graph $G/\Aut(G)$ is connected,
$G$ has some  subtree $T$ containing $\id$ and 
comprising exactly one member of each orbit under $\Aut(G)$. 
Therefore, 
\begin{equation}\label{eq:some}
\si_n=\si_n(v) \text{ for some }v \in V_T,
\end{equation}
where $V_T$ is the vertex-set of $T$.

(a) Let $G\in\sG$ and $m \ge 1$, and write
\begin{equation}\label{eq:L}
L=\max\{M(G),m\}-1.
\end{equation}
By \eqref{eq:silb}, 
there exist $\eta_k=\eta_k(G)$ such that
$0< \eta_k \downarrow 0$ as $k\to\oo$ and
\begin{equation}
\mu(G)^n \le\si_n(G) \le (\mu(G)+\eta_k)^{n},\qq  n \ge k.
\label{eq:bounds}
\end{equation}
Let $G'\in\sG$ be such that $M(G')\le m$, and write $K=K(G,G')$. 
Since $V_T \subseteq S_L(\id_G,G)$ and $S_K(\id_G,G)\simeq S_K(\id_{G'},G')$,
\begin{equation}\label{eq:equal}
\si_{K-L}(G') =\si_{K-L}(G)\qq\text{if } K \ge L.
\end{equation}

Assume $K>L$. 
By \eqref{eq:bounds}--\eqref{eq:equal} and \eqref{eq:sisub}, 
\begin{align*}
\si_{(K-L)t}(G') &\le \si_{K-L}(G')^t\\
&=\si_{K-L}(G)^t \le (\mu(G)+\eta_{K-L})^{(K-L)t},
\qq t \ge 1.
\end{align*}
Take $(K-L)t^\text{th}$ roots and let $t\to\oo$, to
obtain that $\mu(G') \le \mu(G)+ \eta_{K-L}$, and
the claim follows with $\eps_k=\eta_{k-L}$.

(b)
Let $D$, $R$, $\De$ and $G'$ satisfy the given conditions. Since $G'\in\Si_{D,R}$
and $\de_{G'}\le \De$, we have by \eqref{new42} that $M(G')\le m$ 
where $m=(\De^{R+1}-1)/(\De-1)$.
Let $L$ be given by \eqref{eq:L} with this value of $m$.
By \eqref{eq:bounds}--\eqref{eq:equal}
and the forthcoming Proposition \ref{swb},
there exists $B=B(D,R,\De)>0$ such that, for $K >L$,
\begin{align*}
\be(G')^{K-L} &\ge \si_{K-L}(G') e^{-B\sqrt {K-L}}\\
&=\si_{K-L}(G) e^{-B\sqrt {K-L}}
\ge \mu(G)^{K-L} e^{-B\sqrt {K-L}}.
\end{align*}
Therefore,
$$
\be(G') \ge \mu(G)  e^{-B/\sqrt {K-L}}
$$
By Theorem \ref{thm1}, $\mu(G')=\be(G')$,
and \eqref{ineq2} is proved. 
\end{proof}

\begin{proof}[Proof of Theorem \ref{thm2}]
Since $G_n \in \sG_{D,R}$ and $K(G,G_n) \to\oo$, 
by \eqref{new42} the $M(G_n)$
are uniformly bounded, and hence the $\de_{G_n}$ are
also uniformly bounded.
The claim is now immediate by Proposition \ref{prop2}. 
\end{proof}

\section{Proof of Theorem \ref{thm1}: the transitive case}
\label{sec:pf1}

We adapt and extend
the \lq bridge decomposition' approach of Hammersley and Welsh \cite{HW62}, 
which was originally
specific to the hypercubic lattice.
A \emph{distinct partition} $\Pi$ of the integer $n \ge 1$  
is an expression of the form
$n=a_1+a_2 + \dots +a_k$ with 
integers $a_i$ satisfying $a_1>a_2> \dots >a_k >0$ and some $k=k(\Pi) \ge 1$. 
The number $k(\Pi)$ is the \emph{order} of the partition
$\Pi$, and the number of distinct partitions of $n$ is
denoted $P(n)$. We recall
two facts about such distinct partitions.

\begin{lemma}\label{HRlem}
The order $k=k(\Pi)$ and the number $P(n)$ satisfy
\begin{alignat}{2}
k(k+1) &\le 2n \q &&\text{for all distinct partitions $\Pi$ of $n$},\label{pt0}\\ 
\log P(n) &\sim \pi\sqrt{n/3}\quad &&\text{as } n\to\infty.
\label{pt}
\end{alignat}
\end{lemma}

\begin{proof}
The sum of the first $r$ natural numbers 
is $\frac12 r(r+1)$. Therefore, if $r$ satisfies
$\frac12r(r+1)>n$, the order of $\Pi$ is at most $r-1$.
See \cite{HR} for a proof of  \eqref{pt}. 
\end{proof}

Let $G$ be a graph with the given properties, and let $(h,\sH)$ 
be a \ughf\ on $G$.
For the given $(h,\sH)$, and $v \in V$,
we let $b_{n}(v)$ and $c_n(v)$  be the counts of bridges and half-space SAWs 
starting at $v$, \resp, as in Section \ref{sec:saws}.
Recall the constants $d=d(h)$, $r=r(h,\sH)$ given after Definition \ref{def:height}.

\emph{We assume first that $\sH$ acts transitively} (so that, in particular,  $r=0$),
and we add some notes in Section \ref{sec:pf1b} about the quasi-transitive case.

Write $b_n=b_n(v)$ and $c_n=c_n(v)$.
It is elementary that $b_n \le c_n$.
The main idea is to show that $c_n \le  e^{f(n)} b_n$ for
some sublinear function $f$. This is shown by 
`unwrapping' a half-space walk into a bridge, and keeping
track of certain multiplicities. The plan of the proof is as in \cite{HW62}, but the details
are more complicated.

\begin{proposition}\label{hsb}
There exists an absolute constant $A$ such that 
$c_n \le e^{A\sqrt{dn}}b_n$ for $n \ge 1$.
\end{proposition}

The `unwrapping' process involves replacing certain SAWs $\si$
by their images $\g\si$ under certain tailored automorphisms $\g\in\sH$.
The key element in controlling the combinatorics of unwrapping
is the following pair of lemmas, which are based on the unimodularity of $\sH$.

\begin{lemma}\label{lem:msp}
Let $\si$ be a SAW of $G$ with initial (\resp, final) vertex $a$ (\resp, $b$), 
and let $\sH_{u,v} =\{\g\in\sH: \g(u)=v\}$. 
We have that
\begin{equation}\label{new34}
|\sH_{a,v}\si| =  |\sH_{b,v}\si|, \qq v \in V. 
\end{equation}
\end{lemma}

\begin{proof}
For each distinct $\nu\si$, as $\nu$ ranges over $\sH$,
we send one unit of mass from $\nu(a)$ to $\nu(b)$.
Let $m(v,w)$ be the total mass sent from $v$ to $w$. 
By the mass transport principle (see, for example, \cite[eqn (8.4)]{LyP}),
\begin{equation}\label{new35}
\sum_{w\in V} m(v,w) = \sum_{w\in V} m(w,v), \qq v \in V.
\end{equation}
The left side of \eqref{new34} is the mass exiting $v$, and the right side 
is the incoming mass at $v$. By \eqref{new35}, these are equal.
\end{proof}

\begin{lemma}\label{lem:unim2}
Let $a \in V$,  $t \ge 1$, and let $\si=(\si_0,\si_1,\dots,\si_t)$ be a SAW starting at
$\si_0=a$. Let $\si'\in \sH_{b,a}\si$ where $b=\si_t$.
Consider the bipartite graph $B$ with vertex-sets $R:=\sH_{a,a}\si$ 
(coloured red) and $Y:=\sH_{a,a}\si'$ 
(coloured yellow), and an edge between $\si_1\in R$ and  $\si_1'\in Y$
if and only if $\si_1' \in \sH_{b_1,a} \si_1$ where $b_1$ is the endvertex of $\si_1$
other than $a$. 
The graph $B$ is complete bipartite, and the numbers of red and yellow vertices are equal.
\end{lemma}

Write $\si'=\a\si$ where $\a\in\sH_{b,a}$, as in the lemma.
Recall that SAWs have directions:  $\si$ goes from $a=\si_0$ to $b=\si_t$,
and $\si'$ goes from $\a(a)$ to $a$. Later we shall consider the SAW 
obtained by reversing the direction of $\si'$.

\begin{proof}
Let $\si'=\a\si$ where $\a\in\sH_{b,a}$.
To prove that $B$ is \emph{complete} bipartite, it suffices 
that, for $\si_1\in R$,
\begin{equation}\label{gl34}
\sH_{b_1,a} \si_1 = \sH_{a,a}\si',
\end{equation}
where $b_1$ is given in the statement of the lemma.
Let  $\si_1=\g\si$ where $\g\in \sH_{a,a}$, and let $\a_1\in\sH_{b_1,a}$. 
Then 
$$
\a_1\si_1=\a_1\g\si=\a_1\g\a^{-1}\si',
$$
and $\a_1\g\a^{-1}(a) = a$,
whence $\sH_{b_1,a} \si_1 \subseteq \sH_{a,a}\si'$. Conversely, let
$\si_2'\in Y$, say $\si_2'= \be\si'$ with $\be\in\sH_{a,a}$.
Then 
$$
\si_2'=\be\a\si=\be\a\g^{-1}\si_1,
$$
and $\be\a\g^{-1}(b_1)=a$ as required.

The numbers of red and yellow vertices are equal if and only if the degree
of $\si$ equals that of $\si'$. The degree of $\si$ is $|\sH_{b,a}\si|$.
That of $\si'$ is
\begin{align*}
\bigl|\{\g^{-1}\si': \g\in\sH,\  \g^{-1}(c)=a\}\bigr|
&=|\sH_{a,a}\si|,
\end{align*}
where $c$ is the endvertex of $\si'$ other than $a$.
By Lemma \ref{lem:msp} with $v=a$, these are equal, and the lemma is proved.
\end{proof}

\begin{proof}[Proof of Proposition \ref{hsb}]
 Let $n\geq 1$, 
and let $\pi=(\pi_0,\pi_1,\dots,\pi_n)$ be an $n$-step half-space SAW 
starting at $\pi_0=\id$. 
Let $n_0=0$, and for $j \ge 1$,  define $S_j=S_j(\pi)$ and $n_j=n_j(\pi)$ recursively as follows:
\begin{equation*}
S_j=\max_{n_{j-1}\leq m\leq n}(-1)^j\bigl[h(\pi_{n_{j-1}})-h(\pi_m)\bigr],
\end{equation*}
and $n_j$ is the largest value of $m$ at which the maximum is attained. 
The recursion is stopped at the smallest integer $k=k(\pi)$ such that $n_k=n$, 
so that $S_{k+1}$ and $n_{k+1}$ are undefined. 
Note that $S_1$ is the span of $\pi$ and, more generally,
$S_{j+1}$ is the span of the SAW $\ol\pi^{j+1} := (\pi_{n_j},\pi_{n_j+1},\dots,\pi_{n_{j+1}})$. 
Moreover, each of the subwalks $\ol\pi^{j+1}$ is either a bridge or
a reversed bridge. 
We observe that $S_1>S_2>\dots>S_k>0$. 

For a decreasing sequence of $k\ge 2$ positive integers $a_1>a_2>\dots>a_k>0$, 
let $B_n^v(a_1,a_2,\dots,a_k)$ be the set of ($n$-step) half-space walks from $v\in V$ such that 
$k(\pi)=k$, $S_1(\pi)=a_1$, 
$\dots$, $S_k(\pi)=a_k$ and $n_k(\pi)=n$ 
(and hence $S_{k+1}$ is undefined). In particular, $B_n^v(a)$ is the set of $n$-step bridges 
from $v$ with span $a$.  Set $B_n = B_n^\id$.

\begin{lemma}\label{lem:BB}
We have that
\begin{equation}\label{pi'in2}
|B_n(a_1,a_2,\dots,a_k)| \le
 \begin{cases} |B_{n}(a_1+a_2+a_3,a_4,\dots,a_k)| &\text{if } k \ge 3,\\
|B_{n}(a_1+a_2)| &\text{if } k=2.
\end{cases}
\end{equation}
\end{lemma}

\begin{proof}
Let $\pi\in B
_n(a_1,a_2,\dots,a_k)$. We first describe  how to perform surgery on $\pi$
in order to obtain a SAW $\pi'$ satisfying
\begin{equation}\label{pi'in3}
\pi' \in \begin{cases} B_n(a_1+a_2+a_3,a_4,\dots,a_k) &\text{if } k \ge 3,\\
B_n(a_1+a_2) &\text{if } k=2,
\end{cases}
\end{equation}
and then we consider certain multiplicities associated with such mappings $\pi \mapsto \pi'$.

\begin{figure}[htbp]
\centerline{\includegraphics[width=0.7\textwidth]{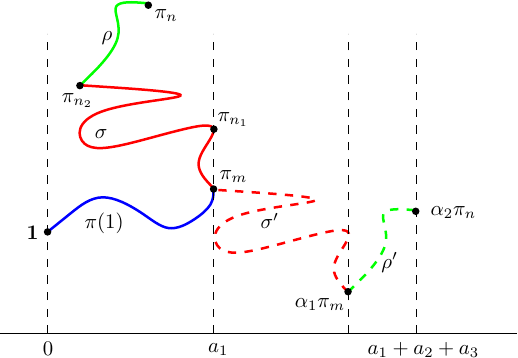}}
    \caption{The solid SAW lies in $B_n(a_1,a_2,a_3)$. We map the red path connecting $\pi_{n_2}$ to
   $\pi_m$, and also the third sub-SAW of $\pi$, thereby obtaining a SAW in $B_{n}(a_1+a_2+a_3)$.
   After translation, the paths are dashed.}
 \label{fig:surgery0}
\end{figure}

The new SAW $\pi'$ is constructed in the following way, 
as illustrated in Figure \ref{fig:surgery0}. 
Suppose first that $k \ge 3$.
\begin{numlist}
\item Let $m=\min\{k: h(\pi_k)=S_1\}$, and let
$\pi(1)$ be the sub-SAW from $\pi_0=\id$ to the vertex $\pi_m$.

\item Let 
$\si:=(\pi_m,\dots,\pi_{n_2})$ and $\rho=(\pi_{n_2}, \dots, \pi_n)$ be the two sub-SAWs of $\pi$
with the given endvertices.  
We find $\a_1\in\sH$ such that $\a_1\pi_{n_2}=\pi_m$.
(This uses the transitive action of $\sH$.)
The concatenation of the two SAWs $\pi(1)$ and 
$\a_1\si$ (reversed) is a SAW, denoted $\pi(2)$, from $\id$ to $\a_1\pi_m$. 
In concluding that $\pi(2)$ is
a SAW, we have made use of Definition
\ref{def:height}(b). Note that $h(\a_1\pi_m) = a_1+a_2$.

\item We next find $\a_2\in\sH$ such that $\a_2(\pi_{n_2})=\a_1(\pi_m)$.
The concatenation of the two SAWs $\pi(2)$ and 
$\a_2\rho$ is a SAW, denoted $\pi'$, from $\id$ to $\a_2\pi_n$.
Note that $S_1(\pi')=a_1+a_2+a_3$.
\end{numlist}

The ensuing $\pi'$ satisfies
$\pi'\in B_n(a_1+a_2+a_3,a_4,\dots,a_k)$.
Note that $\pi$
is not generally reconstructible from knowledge of $\pi'$. 

Suppose now that $k=2$. At Step 2 above, we have that $h(\pi_n) = S_1-S_2$,
so that $n_2=n$ and $\pi'\in B_n(a_1+a_2)$. 

We consider next the multiplicities associated with the map $\pi\mapsto\pi'$:  (i)
for given $\pi$, how many possible choices of $\pi'$ exist, and (ii) for given $\pi'$,
how many pre-images are there? There are two stages to be considered, 
arising in Steps 2 and 3 above. We begin at Step 2.

We assume $k \ge 3$ (the case $k=2$ is similar).
In the above construction, we map $\pi$
to  $\pi'$, and we write
$\pi=(\pi(1),\si,\rho)$ and $ \pi'=(\pi(1),\si',\rho')$ as in Figure \ref{fig:surgery0},
where $\si' =\a_1\si$ and $\rho'=\a_2\rho$. 
Note that, for given $(a_i)$, $\pi$ and $\pi'$ have unique representations in this form.
The representation of $\pi$ is given as above; for given such $\pi'$,
$\pi(2)$ is the shortest sub-SAW from $\id$ to a vertex with height $a_1$, $\rho'$
is the shortest sub-SAW from a vertex with height $a_1+a_2$ to the final endvertex
of $\pi'$, and $\si'$ is the remaining SAW. 
We fix $\pi(1)$ for the moment, and
consider the relationship between $\si'$  and $\si$.  
As usual, $a$ denotes the final endvertex of $\pi(1)$.

Let $a\in V$, $s \ge 1$, and let $\Si(a,s)$ be the set of  SAWs 
$\si=(\si_0,\si_1,\dots,\si_t)$ with $\si_0=a$ and satisfying
$$
h(\si_t)=h(a)-s,\qq h(a)-s\le h(\si_l)\le h(a) 
\q\text{for} \q 1\le l\le t.
$$

Write $\Si:=\Si(a,a_2)$. The SAW $\si$ ranges
over the subset $\ol\Si\subseteq \Si$ comprising all $\si_1\in\Si$
that do not intersect $\pi(1)$ (other than at $a$). 
By Lemma \ref{lem:unim2}, the image $\si'$ lies in the set 
$I:=\bigcup\{\sH_{b_1,a}\si_1: \si_1 \in \Si\}$, where $b_1$ 
denotes the endvertex of $\si_1$ other than $a$.
Now, $\Si$ may be partitioned as $\sP=\{\sH_{a,a}\si: \si\in\Si\}$,
and similarly $I$ may be partitioned as  $\sP'=\{\sH_{a,a}\si': \si'\in I\}$. 
By Lemma \ref{lem:unim2}, there is a one-to-one correspondence
between $\sP$ and $\sP'$, and any corresponding pair $(P_1,P_2) \in\sP\times \sP'$
of sets satisfies $|P_1|=|P_2|$.   It follows that
\begin{equation}\label{gl44}
|\ol\Si| \le |\Si| = |I|.
\end{equation}

A simpler argument is valid for the pair $(\rho, \rho')$ arising at Step 3.
Write $c$ for the endvertex of $\si'$ other than $a$. Then $\rho$ ranges over 
a subset $\ol B\subseteq B_{n-m}^b(a_3,\dots,a_k)$ where $m$ is the length of
$\pi(1)\cup \si$, and its image $\rho'$ lies in $\sH_{b,c}B_{n-m}^b(a_3,\dots,a_k)$.
(As earlier, $b$ denotes the initial vertex of $\rho$.) Hence,
\begin{align}\label{gl45}
|\ol B| \le \bigl|B_{n-m}^b(a_3,\dots,a_k)\bigr|=\bigl|B_{n-m}^c(a_3,\dots,a_k)\bigr|,
\end{align}
since, by transitivity, the last two sets are in one-to-one correspondence.

The set $B_n(a_1,a_2,\dots,a_k)$ is the union of all such $(\pi(1),\si,\rho)$.
Inequality \eqref{pi'in2} follows by \eqref{gl44}--\eqref{gl45}.  
Lemma \ref{lem:BB} is proved.
\end{proof}

Write $\sum_a^{(k,T)}$ for the summation
over all finite integer sequences $a_1>\dots>a_k>0$ with 
given length $k$ and sum $T$. By iteration of \eqref{pi'in2},
\begin{align}\label{new37}
c_n&\le \sum_{T=1}^{dn} \sum_{k=1}^n\sum_a^{(k,T)} \bigl|B_n(a_1,\dots,a_k)\bigr|\\
&\leq \sum_{T=1}^{dn}\sum_{k=1}^n\sum_a^{(k,T)}
  |B_n(a_1+a_2+\dots+a_k)|.\nonumber
\end{align}
By \eqref{pt}, 
there exists an absolute constant $A$ such that
\begin{align}\label{new38}
c_n &\le   
\sum_{T=1}^{dn}\sum_{k=1}^n\sum_a^{(k,T)} b_n
\le e^{A\sqrt{dn}} b_n,
\end{align}
as required for Proposition \ref{hsb}.
\end{proof}

\begin{proposition}\label{swb2}
There exists an absolute constant $B$ such that
$\si_n\leq  e^{B\sqrt{dn}}b_{n+1}$ for $n \ge 1$.
\end{proposition}

\begin{proof}
Let $\pi=(\pi_0,\pi_1,\dots,\pi_n)\in \Si_n$.
Let $H=\min_{0\le i \le n}h(\pi_i)$ and 
$m=\max\{i: h(\pi_i)=H\}$. 
For each such $\pi$, we send one unit of mass
from $\pi_0$ to $\pi_m$, so that the total mass leaving each $v\in V$ is
$\si_n$. By the mass transport principle
(see \cite[Eqn (8.4)]{LyP}), the total mass arriving at each 
$v \in V$ is also $\si_n$. 

Each $\pi\in \Si_n$, seen from 
the vertex $\pi_m$, is the union
of two SAWs $\pi(1)=(\pi_m,\pi_{m-1},\dots,\pi_0)$
and $\pi(2)=(\pi_m,\pi_{m+1},\dots,\pi_n)$.
These two SAWs are vertex-disjoint except at $\pi_m$, 
and $\pi(2)$ is a half-space SAW. We pick $w\in \pd \pi_m$ such that
$h(w)<h(\pi_m)$, and extend $\pi(1)$ by adding $w$ at the start,
thus obtaining a $(m+1)$-step half-space SAW.  
Therefore,
$$
\si_n \le \sum_{m=0}^n c_{m+1}c_{n-m}.
$$

By Proposition \ref{hsb},
$$
\si_n \le \sum_{m=0}^n U_{m+1}U_{n-m},
$$
where $U_k = e^{A\sqrt {dk}}b_k$.
Therefore,
\begin{align*}
\si_n &\le 
\sum_{m=0}^n 
\exp\Bigl(A\sqrt{d(m+1)}+A\sqrt{d(n-m)}\Bigr)b_{m+1}b_{n-m}\\
&\le e^{A\sqrt{2d(n+1)}}b_{n+1},
\end{align*}
by \eqref{eq:b-subadditive} and the fact that $\sqrt x+\sqrt y \leq \sqrt {2x+2y}$.
The claim follows.
\end{proof}

It is trivial that $b_n \le \si_n$, whence $\be \le \mu$.
The reverse inequality follows by Proposition \ref{swb2}, and Theorem 
\ref{thm1} is proved in the transitive case.

\section{Theorem \ref{thm1}: the  quasi-transitive case}\label{sec:pf1b}

We present the further steps needed to prove Theorem \ref{thm1} when
the \ughf\ $(h, \sH)$ is such that $\sH$ acts only \emph{quasi-transitively} on $G$.
Proposition \ref{hsb} is replaced by the following. Recall the maximum
vertex-degree $\de_G$, and the integer $r$ given before Proposition 
\ref{lem:rfin}.

\begin{proposition}\label{hsb2}
There exists $A=A(d,r,\de_G)$, that is non-decreasing in $d$, $r$, and $\de_G$,
such that 
$c_n(v) \le e^{A\sqrt{n}}\be^{n}$ for $n \ge 1$ and $v \in V$.
\end{proposition}

\begin{proof}
We follow the proof of Proposition \ref{hsb}, with the following differences.
Let $o_1,o_2,\dots,o_M$ be representatives of the equivalence classes of $\sH$. By the definition of $r$, we may choose such $o_i$ satisfying
$d_G(o_i,o_j)\le r$ for all $i,j$.
A vertex $v \in V$ is said to have \emph{type} $o_i$ if
$v \in \sH o_i$. 
Lemmas \ref{lem:msp} and \ref{lem:unim2} are replaced by the following.

\begin{lemma}\label{lem:qmsp}
Let $\si$ be a SAW with initial (\resp, final) vertex $a$ (\resp, $b$), and assume
$a$ has type $o_a$ and $b$ has type $o_b$.
There exist  constants  $\xi_1,\xi_2,\dots,\xi_M>0$, depending on $G$ and $\sH$ only, such that,
for $v,w\in V$ with respective types $o_a$ and $o_b$,
\begin{equation}\label{new36}
\frac 1{\xi_a} |\sH_{a,v}\si|= \frac 1{\xi_b}|\sH_{b,w}\si|.
\end{equation}
Furthermore, 
\begin{equation}\label{eq:defmu}
\xi:=\max\bigl\{\xi_i/\xi_j: i,j=1,2,\dots,M\bigr\}
\end{equation}
satisfies
$\xi \le |S_{2r}(\id)|$, and hence
\begin{equation}\label{new40}
\xi\le \frac{\de_G^{2r+1}-1}{\de_G-1}.
\end{equation}
\end{lemma}

\begin{proof}
Equation \eqref{new36} follows by adapting the proof of Proposition \ref{hsb} 
and appealing to \cite[Cor.\ 8.11]{LyP}. By \cite[Thm 8.10]{LyP},
\begin{align*}
\frac{\xi_i}{\xi_j} &= \frac{|\Stab_{o_i}o_j|}{|\Stab_{o_j} o_i|} 
\le |\Stab_{o_i} o_j|\\
&\le |S_r(o_i)| \le |S_{2r}(\id)|,
\end{align*}
by the definition of $r=r(h,\sH)$.
\end{proof}

\begin{lemma}\label{lem:qunim2}
Let $a \in V$,  $t \ge 1$, and let $\si=(\si_0,\si_1,\dots,\si_t)$ be a SAW starting at
$\si_0=a$. Let  $b=\si_t$ and let $b'\in V$ have the same type as $b$.
Fix $\si'\in \sH_{b,b'}\si$. Consider the bipartite graph $B$ with vertex-sets $R:=\sH_{a,a}\si$ 
(coloured red) and $Y:=\sH_{b',b'}\si'$ 
(coloured yellow), and an edge between $\si_1\in R$ and $\si_1'\in Y$
if and only if $\si_1' \in \sH_{b_1,b'} \si_1$ where $b_1$ is the endvertex of $\si_1$
other than $a$. 
The graph $B$ is complete bipartite, and $\xi^{-1}|Y| \le|R| \le \xi|Y|$
where $\xi$ is given by \eqref{eq:defmu}.
\end{lemma}

\begin{proof}
That $B$ is \emph{complete} bipartite follows as in the proof of Lemma \ref{lem:unim2}.
The final inequality holds similarly, but using \eqref{new36}
in place of  \eqref{new34}.
\end{proof}

Let $\pi\in B_n^v(a_1,a_2,\dots,a_k)$. 
We shall perform surgery on $\pi$ to obtain a SAW $\pi'$ satisfying
\begin{equation}\label{pi'in}
\pi' \in \begin{cases} B_{n+\th}^v(a_1+a_2+a_3+\de,a_4,\dots,a_k) &\text{if } k \ge 3,\\
B_{n+\th}^v(a_1+a_2+\de) &\text{if } k=2,
\end{cases}
\end{equation}
for some $\th=\th(\pi)$ and $\de=\de(\pi)$ satisfying
$0\le\th\le 2r$ and $\de\ge 0$. 

\begin{figure}[htbp]
\centerline{\includegraphics[width=0.7\textwidth]{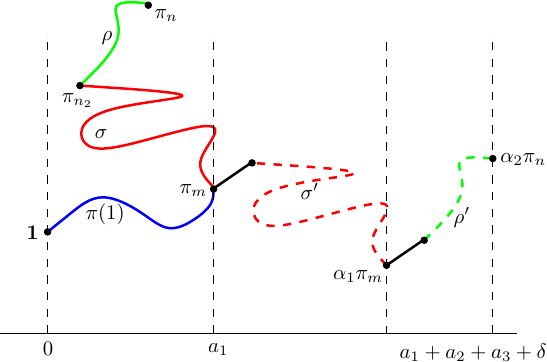}}
     \caption{In the non-transitive case, further SAWs are introduced
   in order to make the required connections. These 
   extra connections are drawn in black, and have lengths
   no greater than $r$.} 
   \label{fig:surgery}
\end{figure}

The  SAW $\pi'$ is constructed as in Steps 1--3 of the previous proof, but with a significant
extra feature. In the construction of $\pi'$  in Section \ref{sec:pf1}, 
the two sub-SAWs $\si$ and $\rho$ are mapped to
new SAWs with the given starting vertices. This cannot generally be done in the non-transitive
setting, since the target starting vertices may not have the 
appropriate types.

Steps 2 and 3 of the preceding proof are thus replaced as follows
(as illustrated in Figure \ref{fig:surgery}).
The SAWs $\nu(\cdot,\cdot)$ of the following are as given after \eqref{eq:defd}.
\begin{numlist}
\item [2$'$.] Let 
$\si:=(\pi_m,\dots,\pi_{n_2})$ and $\rho=(\pi_{n_2}, \dots, \pi_n)$ be the two sub-SAWs of $\pi$
with the given endvertices.  The type of $\pi_m$ (\resp, $\pi_{n_2}$)
is denoted $i$ (\resp, $j$). In particular, $\pi_m=\g(o_i)$ for some $\g\in\sH$.
We map the SAW $\nu(o_i,v_j)$, given after \eqref{eq:defd}, under $\g$ to
obtain a SAW denoted $\nu(\pi_m,\a_1\pi_{n_2})$, where $\a_1\in\sH$ is such that 
$\a_1(\pi_{n_2}) = \g(v_j) $. Note that $\nu(\pi_m,\a_1\pi_{n_2})$ is
 the single point $\{\pi_m\}$ if $i=j$,
and $\a_1(\pi_{n_2})=\pi_m$ in this case.

The union of the three SAWs $\pi(1)$, $\nu(\pi_m,\a_1\pi_{n_2})$, and 
$\si':=\a_1\si$ (reversed) is a SAW, denoted $\pi(2)$, from $v$ to $\a_1\pi_m$. 
Note that $a_1+a_2 \le h(\a_1\pi_m)\le a_1+a_2+r$.

\item [3$'$.]
We next perform a similar construction to connect $\a_1\pi_m$ to an image of the first vertex $\pi_{n_2}$
of $\rho$. These two vertices have types $i$ and $j$, as before, and thus we insert
the SAW  $\nu(\a_1\pi_m,\a_2\pi_{n_2}) :=\a_1\g\nu(o_i,v_j)$, where $\a_2\in\sH$ 
is such that $\a_2(\pi_{n_2}) = \a_1\g(v_j)$. 
The union of the three SAWs $\pi(2)$, $\nu(\a_1\pi_m,\a_2\pi_{n_2})$, and 
$\rho':=\a_2\rho$, is a SAW, denoted $\pi'$, from $\id$ to $\a_2\pi_n$.
\end{numlist}

The resulting SAW  satisfies 
$\pi'\in B_{n+\th}^v(a_1+a_2+a_3+\de,a_4,\dots,a_k)$
for some $0\le \th\le 2r$ and $\de\ge 0$.
In this non-transitive scenario, there are two reasons for which  $\pi$
is not generally reconstructible from knowledge of $\pi'$:
(i) we need to identify the intermediate SAWs $\nu(\cdot)$, and (ii)
the maps $\si\mapsto \si'$ and $\rho\mapsto \rho'$ are not bijections.
Since the intermediate SAWs have lengths no greater
than $r$, issue (i) contributes at worst a factor $(r+1)^2$ when $k \ge 3$
(and a factor $(r+1)$ when $k=2$). 
Issue (ii) is controlled as in the previous proof,  using Lemma
\ref{lem:qunim2} in place of Lemma \ref{lem:qmsp}, thus introducing a factor
$\xi^k$. 

Equation \eqref{new37} is replaced in the quasi-transitive context by
\begin{align*}
c_n(v)&\le \sum_{T=1}^{dn} \sum_{k=1}^n\sum_a^{(k,T)} \bigl|B_n^v(a_1,\dots,a_k)\bigr|\\
&\leq \sum_{T=1}^{dn}\sum_{k=1}^n\sum_a^{(k,T)}
\xi^k (r+1)^k \sum_{s=0}^{kr} b_{n+s}(v).
\end{align*}
As in \eqref{new38}, by \eqref{eq:bnabove}, \eqref{new40}, and that fact that $\be\le\de_G$, 
there exists a constant $A=A(d,r,\de_G)$ with the required 
properties such that
\begin{align*}
c_n(v) &\le   
\sum_{T=1}^{dn}\sum_{k=1}^n\sum_a^{(k,T)} \xi^k(r+1)^k(kr+1)\be^{n+kr+r}\\
&\le \sum_{T=1}^{dn}\sum_{k=1}^n\sum_a^{(k,T)} \be^{n+r}\{\be^r\xi (r+1)\}^{\sqrt{2dn}}\bigl(r\sqrt{2dn}+1\bigr)\\
&\le e^{A\sqrt{n}}\be^{n},
\end{align*}
as required.
\end{proof}

\begin{proposition}\label{swb}
There exists $B=B(d,r,\de_G)>0$, that is non-decreasing in
$d$, $r$, and $\de_G$, such that
$\si_n(v)\le e^{B\sqrt{n}}\be^n$ for $n \ge 1$ and $v \in V$.
\end{proposition}

\begin{proof}
We adapt the proof of Proposition \ref{swb2} to the quasi-transitive setting.
By the mass transport principle, as there,
$$
\sum_i \frac 1{\xi_i}\si_n(o_i) = \sum_j \frac 1{\xi_j} I(o_j),
$$
where $I(v)$ denotes the total mass entering the vertex $v$.
Therefore,
$$
\si_n(v) \le \xi M \sum_{m=0}^n U_{m+1}U_{n-m},
$$
where $U_k = e^{A\sqrt k}\be^k$ and $M=M(\sH)$ is
the number of equivalence classes of $\sH$.
The proof is completed as before.
\end{proof}

\section*{Acknowledgements} 
This work was supported in part
by the Engineering and Physical Sciences Research Council under grant EP/I03372X/1.  
ZL acknowledges support from the Simons Foundation $\#$351813 and the National Science Foundation $\#1608896$.
GRG thanks Russell Lyons for a valuable conversation.
We thank an anonymous referee
for a number of suggestions and two important observations.

\bibliography{sawfinal-23}
\bibliographystyle{amsplain}

\end{document}